\documentclass[12pt]{amsart}

\usepackage{graphics,color}

\usepackage{amssymb,amsmath, amsthm}
\usepackage{graphicx}

\newtheorem{thm}{Theorem}[section]
\newtheorem{theorem}[thm]{Theorem}

\newtheorem{lemma}[thm]{Lemma}

\newtheorem{definition}[thm]{Definition}

\newenvironment{pfofthmFirst}[1]
{\par\vskip2\parsep\noindent{\sc Proof of Theorem\ \ref{Thm:FirstMain}.}}{{\hfill
$\Box$}
\par\vskip2\parsep}

\newenvironment{pfofthmFourth}[1]
{\par\vskip2\parsep\noindent{\sc Proof of Theorem\ \ref{Thm:FourthMain}.}}{{\hfill
$\Box$}
\par\vskip2\parsep}

\usepackage{graphicx}
\usepackage{amsmath}
\usepackage{amsfonts}
\usepackage{amssymb}
\usepackage{color}
\usepackage{verbatim}

\theoremstyle{plain}
\newtheorem{proposition}[thm]{Proposition}
\newtheorem{fact}[thm]{Fact}

\newtheorem*{thm*}{Theorem}

\theoremstyle{remark}

\theoremstyle{definition}

\numberwithin{equation}{section}

\def\1{\mathbf{1}}

\def\E{\mathbb{E}}
\def\P{\mathbb{P}}
\def\R{\mathbb{R}}
\def\N{\mathbb{N}}
\def\Z{\mathbb{Z}}
\DeclareMathOperator{\Var}{Var}

\newcommand{\Reci}[1]{\frac{1}{#1}}
\newcommand{\limn}{\lim\limits_{n\to\infty}}
\def\pin{(\pi_n)}
\def\mbar{\overline{m}}
\begin{document}

\title{Coalescence and meeting times on $n$-block Markov chains
}


\author{Kathleen Lan       \and
        Kevin McGoff 
}




\begin{abstract} 

We consider finite state, discrete-time, mixing Markov chains $(V,P)$, where $V$ is the state space and $P$ is transition matrix. To each such chain $(V,P)$, we associate a sequence of chains $(V_n,P_n)$ by coding trajectories of $(V,P)$ according to their overlapping $n$-blocks. The chain $(V_n,P_n)$, called the $n$-block Markov chain associated to $(V,P)$, may be considered an alternate version of $(V,P)$ having memory of length $n$. Along such a sequence of chains, we characterize the asymptotic behavior of coalescence times and meeting times as $n$ tends to infinity. In particular, we define an algebraic quantity $L(V,P)$ depending only on $(V,P)$, and we show that if the coalescence time on $(V_n,P_n)$ is denoted by $C_n$, then the quantity $\frac{1}{n} \log C_n$ converges in probability to $L(V,P)$ with exponential rate. Furthermore, we fully characterize the relationship between $L(V,P)$ and the entropy of $(V,P)$.


 
\end{abstract}

\maketitle

\section{Introduction}

We consider finite state, discrete time Markov chains, which we denote by $(V,P)$, where $V$ is the state space and $P$ is the stochastic transition matrix. A coalescing random walk on the chain $(V,P)$ is defined as follows. At time $t=0$, place a random walker on each state in $V$. As time evolves, the walkers move independently according to $P$ until any two of them meet, or occupy the same state simultaneously. When two or more walkers meet, they coalesce (become one walker or cluster) and move together thereafter according to $P$. The first time that only one walker remains in the system is the coalescence time of the coalescing random walk associated to $(V,P)$. Note that this time is almost surely finite if and only if $(V,P)$ is mixing (aperiodic and irreducible). We therefore restrict attention to mixing chains.

Coalescing random walks are of interest in their own right and also due to their relationship to the voter model (see \cite{AldousFill} or \cite{Liggett} for an introduction to the voter model) and to certain graph-based algorithms in computer science (see \cite{CEOR}, for example). The voter model is an interacting particle system that can be interpreted as describing the evolution of opinions in a social network and has received considerable attention (see \cite{Durrett} and references therein). Because coalescing random walks are related to the voter model by duality, bounds on the coalescence time for a coalescing random walk may be translated into bounds on the consensus time in the corresponding voter model. 

In this work we study the asymptotic behavior of the coalescence time as the length of the memory in the underlying Markov chain tends to infinity. To be more precise, consider a fixed Markov chain $(V,P)$, and for each $n\geq 1$, define the $n$-block Markov chain $(V_n,P_n)$ as the chain obtained by coding trajectories from $(V,P)$ into blocks of length $n$ (see Definition \ref{Def:nBlockChain} for a precise definition). The chain $(V_n,P_n)$ provides an alternate presentation of $(V,P)$ that has memory of length $n$. 

Such sequences of Markov chains have long been studied in ergodic theory. Furthermore, they have been studied for their connections to data compression at least since the work of Wyner and Ziv \cite{WynerZiv1989}. For a sequence $(V_n,P_n)$ of $n$-block chains, it has been shown that the asymptotic behavior of recurrence times and hitting (waiting) times is governed by the entropy of $(V,P)$ \cite{MartonShields,OrnsteinWeiss1993,Shields1996}. See Section \ref{Sect:PreviousWork} for precise statements and a detailed discussion of connections between the present work and related literature.

Our main results describe the asymptotic behavior of coalescence times and meeting times for sequences of $n$-block chains in terms of an algebraic quantity $L(V,P)$ that only depends on the underlying chain $(V,P)$. Furthermore, we completely characterize the relationship between $L(V,P)$ and the entropy of $(V,P)$, which then characterizes when coalescence and meeting times occur exponentially faster than recurrence and hitting times along a sequence of $n$-block chains.

\subsection{Main results} \label{Sect:MainResults}

Consider a fixed, mixing Markov chain $(V,P)$, and let $(V_n,P_n)$ be the associated $n$-block Markov chain. Let $C_n$ denote the coalescence time of the coalescing random walk on $(V_n,P_n)$. Also, define the matrix $Q$ such that for $u,v$ in $V$, it holds that $Q(u,v)=P(u,v)^2$. Let $\lambda$ be the Perron eigenvalue of $Q$ (also known as the spectral radius of $Q$), and define $L(V,P) = -\log\lambda$. We may now state our first main result regarding the asymptotic behavior of $C_n$ as $n$ tends to infinity: the quantity $\frac{1}{n} \log C_n$ converges in probability to $L(V,P)$ with exponential rate.

\begin{theorem} \label{Thm:FirstMain}
 Suppose $(V,P)$ is a mixing Markov chain. Then for each $\epsilon >0$, there exists $\delta >0$ and $N$ in $\N$ such that if $n \geq N$, then
\begin{equation*}
 \P \biggl ( \biggl| \frac{1}{n} \log C_n - L(V,P) \biggr| > \epsilon \biggr) \leq e^{- \delta n}.
\end{equation*}
\end{theorem}

In our next main result, Theorem \ref{Thm:SecondMain}, we provide additional information about the behavior of $L(V,P)$. In particular, Theorem \ref{Thm:SecondMain} precisely describes the relationship between $L(V,P)$ and another fundamental parameter of $(V,P)$, its entropy. Recall that the entropy of $(V,P)$ may be expressed by the formula
\begin{equation} \label{Eqn:Entropy}
 h(V,P) = - \sum_{u,v \in V} \pi(u) P(u,v) \log P(u,v),
\end{equation}
where $\pi$ is the stationary distribution of the chain. Also recall that the chain $(V,P)$ is called a measure of maximal entropy if $h(V,P) \geq h(V,P')$ for all $|V| \times |V|$ stochastic transition matrices $P'$ such that $P(u,v) >0$ whenever $P'(u,v)>0$.

\begin{theorem} \label{Thm:SecondMain}
 Suppose $(V,P)$ is a mixing Markov chain. Then
\begin{enumerate}
 \item $0 \leq L(V,P) \leq h(V,P)$;
 \item $L(V,P) = 0$ if and only if $(V,P)$ is trivial (\textit{i.e.} $|V|=1$); and
 \item $L(V,P) = h(V,P)$ if and only if $(V,P)$ is a measure of maximal entropy. 
\end{enumerate}
\end{theorem}

In addition to the above results, we obtain several related results regarding meeting times that may be of independent interest. In order to state these additional results, let us define the meeting time of two random walkers. For a mixing Markov chain $(V,P)$ with corresponding $n$-block chain $(V_n,P_n)$ and two states $u,v$ in $V_n$, let $m_n(u,v)$ denote the first meeting time of two independent random walkers on $(V_n,P_n)$ started at $u$ and $v$, respectively. Define
\begin{align*}
 m_n^* & = \max_{u,v \in V_n} \E\bigl(m_n(u,v)\bigr), \, \text{ and } \\
 \mbar_n & = \sum_{u,v \in V_n} \pi_n(u) \, \pi_n(v) \, \E\bigl(m_n(u,v)\bigr),
\end{align*}
where $\pi_n$ is the stationary distribution of $(V_n,P_n)$. The quantity $m_n^*$ captures the maximal expected meeting time of two walkers on $(V_n,P_n)$, and $\mbar_n$ denotes the expected meeting time of two random walkers whose initial positions are chosen at random from the stationary distribution. It turns out that the asymptotic behavior of these quantities is also governed by $L(V,P)$.

\begin{theorem} \label{Thm:ThirdMain}
 Suppose $(V,P)$ is a mixing Markov chain. Then
\begin{equation*}
 \lim_n \frac{1}{n} \log m_n^* = \lim_n \frac{1}{n} \log \mbar_n = L(V,P).
\end{equation*}
\end{theorem}

Our final main result, Theorem \ref{Thm:FourthMain}, gives an almost-sure version of the statement that the asymptotic behavior of meeting times is governed by $L(V,P)$. To make the almost-sure statement precise, let $\mu$ be the measure on $V^{\N}$ induced by the chain $(V,P)$ (see Section \ref{Sect:Preliminaries} for a precise definition). Then for $x,y$ in $V^{\N}$, let
\begin{equation*}
 M_n(x,y) = \inf \{ t \geq 1 : x_t^{t+n-1} = y_t^{t+n-1} \}.
\end{equation*}
Note that choosing $(x,y)$ according to $\mu \times \mu$ makes $M_n(x,y)$ into the meeting time of two random walkers on $(V_n,P_n)$ with initial positions chosen from stationarity.

\begin{theorem} \label{Thm:FourthMain}
 Suppose $(V,P)$ is a mixing Markov chain. Then
\begin{equation*}
 \lim_n \frac{1}{n} \log M_n = L(V,P), \quad \mu \times \mu - \text{a.s.}
\end{equation*}
\end{theorem}

For non-trivial chains, by Theorems \ref{Thm:FirstMain}, \ref{Thm:SecondMain}, and \ref{Thm:ThirdMain}, we see that $\E C_n$, $m_n^*$, and $\mbar_n$ are each exponential in $n$. Moreover, the relevant exponents for these quantities are all equal. The fact that these exponents are all equal is somewhat remarkable, given that full coalescence involves the meeting of exponentially many walkers, whereas a meeting time refers only to the meeting of two walkers. Thus, in terms of the relevant exponents, full coalescence takes about as much time as the meeting of the last two remaining walkers. We note that this general phenomenon has been observed previously in the context of continuous-time chains \cite{Durrett,Oliveira2013}. 

\subsection{Relation to previous work} \label{Sect:PreviousWork}

The expected coalescence time has been studied for certain types of chains, including random walks on the torus in ${\mathbb{Z}}^d$ \cite{Cox}, on some random graphs \cite{CFR2009}, and on certain general classes of graphs  \cite{CEOR,Oliveira2013}. In all of these instances, one considers a sequence of chains in which the cardinality of the state space tends to infinity, and one investigates the asymptotic behavior of the coalescence times along the sequence. Since the expected coalescence time is generally quite difficult to calculate exactly, methods for estimating its order of magnitude are often studied instead. Typically, other parameters of the Markov chain, such as hitting times, meeting times, or spectral gaps, are used to give bounds on the asymptotic behavior of the expected coalescence time. 

In the discrete-time setting, it was recently shown in \cite{CEOR} that for a large class of undirected, connected graphs $G$, the expected coalescence time for the lazy random walk on $G$ can be bounded above by a constant multiple of $\frac{n}{v(1-\lambda_2)}$, where $n$ is the number of vertices, $v$ is a measure of the variability of the degree distribution, and $1-\lambda_2$ is the spectral gap of the chain.

Additional work in both the discrete- and continuous-time settings has been devoted to studying the following question of Aldous and Fill \cite{AldousFill}: does there exists a universal constant $K$ such that $\E(C) \leq K H$, where $H$ is the maximal expected hitting time for the chain. In discrete-time, the question has been answered in the affirmative for random walks on $r$-regular graphs by \cite{CFR2009}. In continuous time, the question was originally answered affirmatively for tori in $\Z^d$ by \cite{Cox}, and more recently it was answered affirmatively for reversible chains by \cite{Oliveira2012}.

Furthermore, recent work of Oliveira \cite{Oliveira2013} gives estimates on the expected coalescence times for large classes of continuous time chains. The bounds in that work are stated in terms of the mixing time and the meeting time of the chain. These results show that for chains with fast mixing, the expected coalescence time is bounded above by a constant multiple of the expected meeting time of two walkers.

In contrast to the general results cited above, our results deal with specific sequences of chains. In exchange for this specificity, we are able to find exact representations for the exponential order of magnitude of the corresponding meeting times and coalescence times. Furthermore, we go beyond the expected coalescence and meeting times; indeed, we obtain results in probability and almost surely.

Let us discuss the basic structure of sequences of $n$-block Markov chains in order to place them in the broader context of other well-studied chains. Consider the non-trivial case $|V|>1$. First, note that the number of states in $V_n$ grows exponentially in $n$. Second, in the associated directed graph (with vertex set $V_n$ and an edge from $u$ to $v$ whenever $P_n(u,v)>0$), both the in-degree and the out-degree of any vertex is uniformly bounded in $n$. Lastly, it's not difficult to see that the mixing time of $(V_n,P_n)$ is $n+C$ for some constant $C$ depending only on $(V,P)$. Thus, if one would like to think of these chains as certain random walks on the associated graphs, then these graphs are sparse (bounded degrees), but the chains mix relatively quickly (mixing time is logarithmic in the number of vertices).

Lastly, let us mention a line of work in ergodic theory and information theory that also considers asymptotic properties of Markov chains in the long memory limit. With the same notation as above, for $x$ in $V^{\N}$, define
\begin{equation*}
R_n(x) = \inf \{ t > 1 :  x_t^{t+n-1} = x_1^n\}.
\end{equation*}
If $x$ is chosen according to the measure $\mu$ defined by a Markov chain $(V,P)$, then $R_n$ may be viewed as the first return time of a random walk on $(V_n,P_n)$ to its initial state, started from stationarity. This quantity was originally studied by Wyner and Ziv \cite{WynerZiv1989} in the context of data compression. They showed that $n^{-1} \log R_n$ converges in probability to the entropy of the process, $h(V,P)$. Later Ornstein and Weiss \cite{OrnsteinWeiss1993} proved that for any ergodic process on a finite alphabet,  $n^{-1} \log R_n$ converges almost surely to the entropy of the process. 

Similarly, Wyner and Ziv considered the $n$-block waiting time (or hitting time): for $x,y$ in $V^{\N}$, let
\begin{equation*}
W_n(x,y) = \inf \{ t \geq 1 : x_1^n = y_t^{t+n-1} \}.
\end{equation*}
They show in \cite{WynerZiv1989} that $n^{-1} \log W_n$ converges in probability to the entropy $h(V,P)$. Several other authors generalized this result to a wider class of processes \cite{NobelWyner,Shields1993} and obtained the almost-sure version of the statement \cite{MartonShields,Shields1996}.

Notice that the relevant exponents for recurrence times and waiting times (hitting times) are given by $h(V,P)$, whereas the exponent for meeting times and coalescence times is given by $L(V,P) \leq h(V,P)$. Furthermore, we can say exactly when these quantities are equal: whenever $(V,P)$ is a measure of maximal entropy. As a consequence, we obtain that whenever $(V,P)$ is not a measure of maximal entropy, meeting times and coalescence times occur exponentially faster than waiting times or return times along the sequence of $n$-block chains.




In light of the progress made in studying return times and hitting times for more general classes of processes than Markov chains, it might be interesting to study meeting times, and possibly coalescence times, in the context of more general processes than Markov chains, but we leave this direction for future work. 

\section{Preliminaries} \label{Sect:Preliminaries}

We consider a finite set $V$ and a stochastic transition matrix $P$ indexed by $V$. That is, for each pair $(u,v)$ in $V$, we have $P(u,v) \geq 0$, and furthermore for each $u$ in $V$, it holds that
\begin{equation*}
 \sum_{v \in V} P(u,v) =1.
\end{equation*}
We refer to any such pair $(V,P)$ as a Markov chain with state space $V$ and transition matrix $P$. We say that the chain is \textit{non-trivial} if $|V|>1$, and the chain is \textit{mixing} if there exists $n \geq 1$ such that $P^n >0$. 

Suppose $(V,P)$ is a mixing Markov chain. By the Perron-Frobenius Theorem, there exists a unique stochastic left eigenvector of $P$ with eigenvalue $1$, and we denote this eigenvector by $\pi$. We refer to $\pi$ as the stationary distribution of the chain. For words $u = u_1 \dots u_n$ in $V^n$, we denote by $u_i^j$ the subword $u_i \dots u_j$. Also, we do not distinguish between the word $u$ and the set 
\begin{equation*}
 \{ x \in V^{\N} : x_1^n = u \}.
\end{equation*}
 Let $\mu$ denote the probability measure on $V^{\N}$ characterized by the following condition: for each $u$ in $V^n$,
\begin{equation} \label{Def:mu}
 \mu(u) = \pi(u_1) \prod_{j=1}^{n-1} P(u_j, u_{j+1}).
\end{equation}

\begin{definition} \label{Def:nBlockChain}
Suppose $(V,P)$ is a mixing Markov chain. For each $n \geq 1$, we define the $n$-block chain associated to $(V,P)$ to be the chain $(V_n,P_n)$ such that
\begin{equation*}
V_n = \{ u \in V^n : \mu(u) >0 \},
\end{equation*}
and for $u,v$ in $V_n$,
\begin{equation*}
 P_n(u,v) = \left\{ \begin{array}{ll}
                     P(u_n,v_n), & \text{ if } u_2^n = v_1^{n-1} \\
                         0, & \text{ otherwise}.
                    \end{array}
            \right.
\end{equation*}
\end{definition}
Note that the $n$-block chain $(V_n,P_n)$ is mixing since we have assumed that $(V,P)$ is mixing. Let $\pi_n$ denote the stationary distribution of $(V_n,P_n)$. One may check easily that $\pi_n(u) = \mu(u)$ for any word $u$ in $V_n$. For notation, we define
\begin{equation} \label{Def:Deltan}
 \Delta_n = \sum_{u \in V_n} \pi_n(u)^2 = \sum_{u \in V_n} \mu(u)^2 = \mu \times \mu \Bigl( x_1^n = y_1^n \Bigr).
\end{equation}

For each $n$, let $\P$ denote the probability measure corresponding to the coalescing random walk on $(V_n,P_n)$ (omitting the dependence of $\P$ on $(V,P)$ and $n$). We denote by $\E$ and $\Var$ the expectation and variance operators with respect to $\P$, respectively. When taking expectation with respect to another measure $\nu$, possibly on another probability space, we use the notation $\E_{\nu}$.

Let $C_n$ denote the full coalescence time of the coalescing random walk on $(V_n,P_n)$. For $u,v$ in $V_n$, let $m_n(u,v)$ be the random variable giving the first meeting time of two random walkers started at $u$ and $v$. Define
\begin{align*}
 m_n^* & = \max_{u,v \in V_n} \E\bigl( m_n(u,v)  \bigr) \\
 \mbar_n & = \sum_{u,v \in V_n} \pi_n(u) \,  \pi_n(v) \, \E\bigl( m_n(u,v) \bigr).
\end{align*}

\subsection{Thermodynamic formalism for Markov chains} \label{Sect:ThermFormalism}

The proofs of Theorem \ref{Thm:EqEntropy} and Proposition \ref{Prop:DeltanLimit} both appeal to the thermodynamic formalism for dynamical systems \cite{Bowen,Ruelle,Walters}. The survey \cite{BoylePetersen}, which directly addresses Markov chains, contains all the information from the thermodynamic formalism needed in this work. Here we state the relevant facts. 
The topological support of $\mu$ is
\begin{equation*}
 X = \{ x \in V^{\N} : \forall n, \, \mu(x_1^n)>0\}.
\end{equation*}
Let $\sigma : X \to X$ denote the left-shift map defined by $\sigma(x)_n = x_{n+1}$. The measure $\mu$, defined by (\ref{Def:mu}), is also characterized by a certain variational property, which we now discuss. For any continuous function $f : X \to \R$, a subadditivity argument implies that the following limit exists:
\begin{equation*}
 \mathcal{P}(f) = \lim_n \frac{1}{n} \log \sum_{u \in V_n} \exp\Biggl( \sum_{j = 0}^{n-1} f \circ \sigma^j (x(u)) \Biggr),
\end{equation*}
where $x(u)$ is any point in $X$ satisfying $x(u)_1^n = u$. The quantity $\mathcal{P}(f)$ is called the pressure of $f$. For any $\sigma$-invariant measure $\nu$ on $X$, let $h(\nu)$ denote that measure-theoretic entropy of $\nu$. The variational principle states that
\begin{equation*}
 \mathcal{P}(f) = \sup_{\nu} \biggl\{ h(\nu) + \int f d\nu \biggr\},
\end{equation*}
where the supremum runs over all $\sigma$-invariant Borel probability measures $\nu$ on $X$. For any locally constant function $f: X \to \R$, there exists a unique $\sigma$-invariant Borel probability measure $\mu_f$ on $X$ such that
\begin{equation*}
 \mathcal{P}(f) = h(\mu_f) + \int f d\mu_f.
\end{equation*}
The measure $\mu_f$ is called the equilibrium state for $f$. The measure $\mu$ corresponding to the Markov chain $(V,P)$ (defined by (\ref{Def:mu})) is characterized as the equilibrium state for the function $g : X \to \R$ given by $g(x) = \log P(x_1, x_2)$. Further, the equilibrium state for the constant zero function, denoted $\mu_0$, is called the \textit{measure of maximal entropy} on $X$. 

We also need the following result of Parry and Tuncel:
\begin{theorem}[\cite{ParryTuncel}] \label{Thm:ParryTuncel}
Suppose $X$ is the topological support of an irreducible Markov chain and $f,g: X \to \R$ are locally constant. Then the following are equivalent:
\begin{enumerate}
 \item $\mu_f = \mu_g$;
 \item there exists a continuous function $k : X \to \R$ and a constant $c$ such that $f = g + k - k \circ \sigma +c$.
\end{enumerate}
\end{theorem}

Lastly, let us mention the following fact, which will be used in the proof of Theorem \ref{Thm:SecondMain}. For reference, see \cite{BoylePetersen}.

\begin{fact} \label{Fact:PandL} Suppose $(V,P)$ is a mixing Markov chain. Let $Q$ be the $|V| \times |V|$ matrix defined, for $u,v$ in $V$, by $Q(u,v) = P(u,v)^2$. Let $\lambda$ be the Perron eigenvalue of $Q$, and let $f : X \to \R$ be the function $f(x) = \log Q(x_1,x_2)$. Then $\mathcal{P}(f) = \log \lambda$.
\end{fact}
%

\section{Proofs}

\subsection{Bounds for general chains}

In this section, we consider a mixing Markov chain $(V,P)$, and we denote by $C$ the coalescence time on $(V,P)$. For $u,v$ in $V$, we let $m(u,v)$ denote the meeting time of two random walkers on $(V,P)$ started at $u$ and $v$, respectively. Finally, we let $m^* = \max_{u,v \in V} \E(m(u,v))$. The following three lemmas hold for all such chains. We do not claim that these results are new. In fact, the proofs are easy adaptations of the corresponding results for continuous-time chains (see \cite{AldousFill}), and we only include them for completeness.

The following lemma gives an exponential tail bound on the hitting time of a set in terms of the maximal expected hitting time of the set over all initial positions for the chain.

\begin{lemma} \label{Lemma:HittingTimeBound}
 Suppose $(V,P)$ is a mixing Markov chain with $B \subset V$. Let $T_B$ be the first hitting time of $B$ by a random walker on $(V,P)$. For a probability measure $\nu$ on $V$, let $\P_{\nu}$ denote the distribution of a random walker on $(V,P)$ with initial distribution $\nu$, and let $\E_{\nu}$ denote expectation with respect to $\P_{\nu}$. If $\delta_v$ is the point mass at $v$ in $V$, then we set $\P_v = \P_{\delta_v}$ and $\E_v = \E_{\delta_v}$. Define $m_B = \max_{v \in V} \E_v(T_B)$. Then for any probability measure $\nu$ on $V$, we have that
\begin{equation*}
\P_{\nu}(T_B > t) \leq \exp\biggl( - \frac{t}{e m_B}\biggr).
\end{equation*}
\end{lemma}
\begin{proof}
First, observe that for any probability distribution $\nu$ on $V$, $s>0$, and integer $k \geq 1$, there exists a probability distribution $\theta$ on $V$ such that
\begin{align*}
 \P_{\nu}(T_B > k s \mid T_B > (k-1)s) = \P_{\theta}(T_B > s) \leq \max_{v \in V} \P_v(T_B > s),
\end{align*}
and therefore by Markov's inequality, we have
\begin{equation} \label{Eqn:CondIneq}
 \P_{\nu}(T_B > ks \mid T_B > (k-1)s) \leq \frac{ \max_{v \in V} \E_v(T_B) }{s} = \frac{m_B}{s}.
\end{equation}
We now prove by induction on $k$ that for any probability distribution $\nu$ on $V$ and $s>0$, it holds that
\begin{equation} \label{Eqn:InductionStatement}
 \P_{\nu}(T_B > ks) \leq \biggr(\frac{m_B}{s}\biggr)^k.
\end{equation}
Note that for $k=1$, we know from (\ref{Eqn:CondIneq}) that for any $\nu$ and $s>0$,
\begin{equation*}
\P_{\nu}(T_B > s) \leq \P_{\nu}( T_B > s \mid T_B >0) \leq \frac{m_B}{s},
\end{equation*}
which establishes the base case. Now suppose for induction that (\ref{Eqn:InductionStatement}) holds for some $k$. Then by (\ref{Eqn:InductionStatement}) and (\ref{Eqn:CondIneq}),
\begin{align*}
 \P_{\nu}(T_B > (k+1)s) & = \P_{\nu}(T_B > ks) \P_{\nu}( T_B > (k+1)s \mid T_B > ks) \\
 & \leq \biggl( \frac{m_B}{s} \biggr)^k \biggl( \frac{m_B}{s} \biggr) \\
 & = \biggl( \frac{m_B}{s} \biggr)^{k+1},
\end{align*}
which completes the induction.

Now let $t = ks$. Rewriting (\ref{Eqn:InductionStatement}) gives
\begin{equation} \label{Eqn:Rewrite}
\P_{\mu}( T_B > t) \leq \biggl( \frac{m_B}{s} \biggr)^{\frac{t}{s}}.
\end{equation}
Choosing $s = e m_B$, we obtain
\begin{equation*}
\P_{\mu}(T_B > t) \leq \exp\biggl( \frac{-t}{e m_B} \biggr),
\end{equation*}
as desired.
\end{proof}

The following lemma provides an exponential tail bound for the meeting time of any two random walkers in a chain in terms of the maximal expected meeting time over all starting positions. The proof relies on the simple fact that the meeting time of two walkers corresponds to the hitting time of the diagonal in the product chain. After making this connection, we apply Lemma \ref{Lemma:HittingTimeBound}.

\begin{lemma} \label{Lemma:TailBound}
 Suppose $(V,P)$ is a mixing Markov chain with maximal expected meeting time $m^*$. Then for any $u,v$ in $V$, it holds that
\begin{equation*}
 \P\bigl( m(u,v) > t \bigr) \leq \exp\biggl( - \frac{t}{e m^*} \biggr).
\end{equation*}
\end{lemma}
\begin{proof}
Let $(V \times V, P_{\times})$ denote the product chain given by $P_{\times}((a,b),(c,d)) = P(a,c) \, P(b,d)$, and let $\P_{(u,v)}^{\times}$ denote the probability measure corresponding to the chain $(V \times V, P_{\times})$ started at $(u,v)$. Note that $\P(m(u,v)>t) = \P_{(u,v)}^{\times}( T_D > t)$, where $T_D$ is the hitting time of the diagonal $D = \{ (v,v) : v \in V\}$. Also, $m^* = m_D$. By Lemma \ref{Lemma:HittingTimeBound} applied to $(V \times V, P_{\times})$, we obtain
\begin{equation*}
 \P(m(u,v)>t) = \P_{(u,v)}^{\times}( T_D > t) \leq \exp\biggl( -\frac{t}{e  m_D} \biggr) = \exp\biggl( - \frac{t}{e m^*} \biggr).
\end{equation*}
\end{proof}

The following lemma, whose proof is trivial, states that the meeting time of two random walkers, started from arbitrary initial positions, is less than or equal to the full coalescence time.

\begin{lemma} \label{Lemma:MeetingTimeLB}
Suppose $(V,P)$ is a mixing Markov chain with coalescence time $C$ and maximal expected meeting time $m^*$. For $u,v$ in $V$, denote by $m(u,v)$ the meeting time of random walkers on $(V,P)$ started at $u$ and $v$, respectively. Then for any $u,v$ in $V$,
\begin{equation*}
m(u,v) \leq C, 
\end{equation*}
and therefore
\begin{equation*}
m^* \leq \E(C).
\end{equation*}
\end{lemma} 
\begin{proof}
Let $u,v$ be in $V$. Coalescence implies that all walkers have met, including the walkers started at $u$ and $v$, respectively. Therefore $m(u,v) \leq C$ with probability one.
\end{proof}

\subsection{Moment bounds} \label{Sect:Moments}

At this point, we turn to sequences of $n$-block chains associated to a mixing Markov chain $(V,P)$. The following two lemmas obtain bounds on the first two moments of the full coalescence time on such sequences.
These bounds are essentially consequences of the tail bound given by Lemma \ref{Lemma:TailBound}.

\begin{lemma} \label{Lemma:FirstMomentBound}
 Suppose $(V,P)$ is a mixing Markov chain with $n$-block chain $(V_n,P_n)$. Then there exists a constant $K>0$ such that for each $n$ in $\N$, it holds that
\begin{equation*}
 \E(C_n) \leq K n \, m_n^*.
\end{equation*}
\end{lemma}
\begin{proof}
Let $1$ be an arbitrary vertex in $V_n$. Note that if all pairs of walkers have met at time $t$, then $C_n \leq t$, and so $C_n \leq \max_{u \in V_n} m_n(1,u)$. Then
\begin{align} \label{Eqn:Reds}
\begin{split}
\P(C_n > t) & \leq \P \biggl( \max_{u \in V_n} m_n(1,u) > t \biggr) \\ 
& = \P\biggl( \bigcup_{u \in V_n} \{ m_n(1,u) > t\} \biggr) \\
& \leq \sum_{u \in V_n} \P(m_n(1,u) > t).
\end{split}
\end{align}
Then by tail-sum formula, Equation (\ref{Eqn:Reds}), and Lemma \ref{Lemma:TailBound}, we have that
\begin{align} \label{Eqn:Cardinals}
\begin{split}
\E(C_n) & = \sum_{t > 0} \P(C_n > t) \\
 & \leq \sum_{t > 0} \sum_{u \in V_n} \P( m_n(1,u) > t) \\
 & \leq \sum_{t > 0} \min\biggl(1,  |V_n| \exp\biggl( - \frac{t}{e m_n^*} \biggr) \biggr). 
\end{split}
\end{align}
By calculus, we note that
\begin{equation} \label{Eqn:Bulls}
\sum_{t > 0} \min\biggl(1,  A \exp\biggl( - a t \biggr) \biggr) \leq \frac{1}{a} \bigl(\log A+1 \bigr).
\end{equation}
Combining (\ref{Eqn:Cardinals}) and (\ref{Eqn:Bulls}), we obtain that
\begin{equation} \label{Eqn:Grasshoppers}
\E(C_n) \leq e \bigl( \log|V_n| + 1\bigr) \, m_n^*.
\end{equation}
Notice that $|V_n| \leq |V|^n$. Set $K = e(\log |V| +1)$. Then (\ref{Eqn:Grasshoppers}) implies that
\begin{equation*}
\E(C_n) \leq K n m_n^*,
\end{equation*}
as desired.
\end{proof}

In similar fashion, we now obtain a bound on the second moment of $C_n$ along a sequence of $n$-block chains.

\begin{lemma} \label{Lemma:SecondMoment}
 Suppose $(V,P)$ is a mixing Markov chain with $n$-block chain $(V_n,P_n)$. Then there exists $K>0$ such that for each $n$ in $\N$, it holds that
\begin{equation*}
 \E \bigl( C_n^2 \bigr) \leq K \, n^2 \, \E (C_n)^2 .
\end{equation*}
\end{lemma}
\begin{proof}
Select an arbitrary vertex in $V_n$ and denote it by $1$. Since full coalescence implies that the walker started at vertex $1$ has met every other walker, we have that
\begin{equation*}
 C_n \leq \max_{v \in V_n} m_n(1,v).
\end{equation*}
Then for all $t>0$,
\begin{align} \label{Eqn:Nationals}
 \begin{split}
\P(C_n > t) & \leq \P\biggl(\max_{v \in V_n} m_n(1,v) >t \biggr) \\
 & \leq \P \Biggl( \bigcup_{v \in V_n} \{m_n(1,v) > t\} \Biggr) \\
 & \leq \sum_{v \in V_n} \P(m_n(1,v) >t).  
 \end{split}
\end{align}
Thus, by (\ref{Eqn:Nationals}) and Lemma \ref{Lemma:TailBound}, 
\begin{equation} \label{Eqn:Braves}
 \P\bigl(C_n > \sqrt{t}\bigr) \leq \sum_{v \in V_n} \P\bigl(m_n(1,v) > \sqrt{t}\bigr) \leq |V_n| \exp\biggl(-\frac{\sqrt{t}}{ e m_n^*}\biggr).
\end{equation}
Hence, by Tail-Sum formula and (\ref{Eqn:Braves}), we have that
\begin{align} \label{Eqn:Rockies}
\begin{split}
 \E\Bigl(C_n^2\Bigr) &  = \sum_{t > 0} \P \Bigl(C_n^2 > t \Bigr) \\
 & = \sum_{t >0} \P \Bigl(C_n > \sqrt{t}\Bigr) \\
 & \leq \sum_{t > 0} \min\biggl(1,|V_n| \exp\biggl(-\frac{\sqrt{t}}{e m_n^*}\biggr) \biggr).
\end{split}
\end{align}
By calculus, we see that
\begin{equation} \label{Eqn:Calculus}
 \sum_{t > 0} \min\Bigl(1,A e^{-a \sqrt{t}}\Bigr) \leq \biggl( \frac{1}{a^2} \biggr) \bigl( (\log A)^2 + 2 \log A +1\bigr).
\end{equation}
Combining (\ref{Eqn:Rockies}) and (\ref{Eqn:Calculus}), we obtain
\begin{equation*}
 \E \bigl(C_n^2 \bigr) \leq e^2 (m_n^*)^2 \bigl( (\log |V_n|)^2 + 2 \log |V_n| + 1 \bigr).
\end{equation*}
Note that $\log|V_n| \leq n \log|V|$. Setting $K = e^2( 3 \log|V|  +1)$, we have shown that 
\begin{equation*}
 \E\bigl(C_n^2\bigr) \leq K n^2 (m_n^*)^2.
\end{equation*}
Note that $\E(C_n) \geq m_n^*$ by Lemma \ref{Lemma:MeetingTimeLB}, and therefore
\begin{equation*}
 \E\bigl(C_n^2\bigr) \leq K \, n^2 \, \E(C_n)^2,
\end{equation*}
as desired.
\end{proof}

\subsection{Expected coalescence and meeting times}

The goal of this section is to prove Theorem \ref{Thm:Expectations}, which gives the exponential growth rates of expected coalescence and meeting times.

Recall the following notations from Section \ref{Sect:Preliminaries}. Let $(V,P)$ be a mixing Markov chain with $n$-block chain $(V_n,P_n)$. We denote by $\mu$ the measure on $V^{\N}$ corresponding to the chain $(V,P)$, and we let $X$ denote the topological support of $\mu$. We have
\begin{equation*}
 \Delta_n = \sum_{u \in V_n} \mu(u)^2 .
\end{equation*}
Also, we define the matrix $Q$ such that for $u,v$ in $V$, it holds that $Q(u,v)=P(u,v)^2$. Let $\lambda$ be the Perron eigenvalue of $Q$ (also known as the spectral radius of $Q$), and define $L(V,P) = -\log\lambda$.
These notations, along with some of the facts in Section \ref{Sect:ThermFormalism}, are used throughout this section.

\begin{proposition} \label{Prop:DeltanLimit}
Suppose $(V,P)$ is a mixing Markov chain. Then
 \begin{equation*}
  \lim \frac{1}{n} \log \Delta_n = - L(V,P).
 \end{equation*}
\end{proposition}
\begin{proof}
Let $f : X \to \R$ be the function $f(x) = 2 \log P(x_1, x_2)$. Then 
\begin{align}
\begin{split} \label{Delta_nExp}
 \Delta_n & = \sum_{u \in V_n} \pi_n(u)^2 \\
 & = \sum_{u \in V_n} \mu(u)^2 \\
 & = \sum_{v \in V} \mu(v)^2 \sum_{\substack{u \in V_n \\ u_1 = v}} \mu(u \mid u_1=v)^2 \\
 & = \sum_{v \in V} \mu(v)^2 \sum_{\substack{u \in V_n \\ u_1 = v}} \prod_{j=1}^{n-1} P(u_j,u_{j+1})^2 \\
 & = \sum_{v \in V} \mu(v)^2 \sum_{\substack{u \in V_n \\ u_1 = v}} \exp\biggl( \sum_{j=1}^{n-1} 2 \log P(u_j,u_{j+1}) \biggr) \\
 & = \sum_{v \in V} \mu(v)^2 \sum_{\substack{u \in V_n \\ u_1 = v}} \exp\Biggl( \sum_{j=0}^{n-2} f \circ \sigma^j (x_u) \Biggr),
\end{split}
\end{align}
where $x_u$ denotes any point in $X$ such that $x_1^{n} = u$.
Define
\begin{equation*}
 Q_n(f) = \sum_{u \in V_n} \exp\Biggl( \sum_{j=0}^{n-2} f \circ \sigma^j (x_u) \Biggr).
\end{equation*}
Set $C_1 = \min_{v \in V} \mu(v)^2$ and $C_2 = \max_{v \in V} \mu(v)^2$. Note that $C_1 >0$ since $(V,P)$ is mixing. Then by (\ref{Delta_nExp}), for all $n$, we have
\begin{equation} \label{Eqn:Marlins}
 C_1 Q_n(f) \leq \Delta_n \leq C_2 Q_n(f).
\end{equation}
By standard results in the thermodynamic formalism (see Section \ref{Sect:ThermFormalism}), there exists a constant $\mathcal{P}(f)$, called the pressure of $f$, such that
\begin{equation} \label{Eqn:Pirates}
\lim_n \frac{1}{n} \log Q_n(f) = \mathcal{P}(f).
\end{equation}
Furthermore, by Fact \ref{Fact:PandL}, $\mathcal{P}(f) = \log \lambda = - L(V,P)$. Then by (\ref{Eqn:Marlins}) and (\ref{Eqn:Pirates}), we have that
\begin{equation*}
 \lim_n \frac{1}{n} \log \Delta_n = \lim_n \frac{1}{n} \log Q_n(f) = \mathcal{P}(f) =  -L(V,P).
\end{equation*}
\end{proof}

Having established the exponential rate of decay of $\Delta_n$ in Proposition \ref{Prop:DeltanLimit}, we now estimate meeting times and coalescence times in terms of $\Delta_n$. We begin these estimates with a lemma.

\begin{lemma} \label{Lemma:MeetingTimeBound}
Suppose $(V,P)$ is a mixing Markov chain. For any $\epsilon >0$, there exists $T>0$ such that for any $n$ and any $u,v$ in $V_n$, 
\begin{equation*}
 \mu \times \mu \biggl( x_{n+T}^{2n+T-1} = y_{n+T}^{2n+T-1} \Bigl| \, x_1^n = u, \; y_1^n = v  \biggr) \geq (1-\epsilon) \Delta_n.
\end{equation*}
\end{lemma}
\begin{proof}
 Let $(V\times V,P_{\times})$ be the direct product Markov chain, defined for $a,b,c,d$ in $V$ by
\begin{equation*}
P_{\times}((a,b),(c,d)) = P(a,c) P(b,d).
\end{equation*}
Since $(V,P)$ is mixing, the chain $(V \times V, P_{\times})$ is also mixing. Also, the stationary distribution for $(V \times V, P_{\times})$ is given by $\pi_{\times}((a,b)) = \pi(a)\pi(b)$. By the convergence theorem for mixing Markov chains (see \cite{AldousFill} or \cite{LPW}), for any $a,b,c,d$ in $V$, we have that
\begin{equation*}
\lim_n P_{\times}^n((a,b),(c,d)) = \pi(c)\pi(d).
\end{equation*}
Therefore, for $\epsilon >0$, there exists $T>0$ such that
\begin{equation} \label{Eqn:Mets}
\min_{a,b,c,d \in V} \frac{P_{\times}^T((a,b),(c,d))}{\pi(c)\pi(d)} \geq 1-\epsilon.
\end{equation}
Then for $u,v,w$ in $V_n$, the Markov property and (\ref{Eqn:Mets}) give that
\begin{align*}
\mu \times \mu \biggl(  x_{n+T}^{2n+T-1} = y_{n+T}^{2n+T-1}  = w \biggl| & \, x_1^n = u, \, y_1^n = v \biggr) \\ & = P_{\times}^T((u_n,v_n),(w_1,w_1)) \frac{\mu(w)^2}{\pi(w_1) \pi(w_1)} \\ 
 & \geq (1-\epsilon) \mu(w)^2.
\end{align*}
Summing over $w$ in $V_n$, we obtain
\begin{align*}
\mu \times \mu \biggl(  x_{n+T}^{2n+T-1} = y_{n+T}^{2n+T-1}  \biggl| & \, x_1^n = u, \, y_1^n = v \biggr) \\ 
& \geq (1-\epsilon) \sum_{w \in V_n} \mu(w)^2 \\
& = (1-\epsilon) \Delta_n.
\end{align*}
\end{proof}

The following proposition relates expected meeting times to $\Delta_n$.

\begin{proposition} \label{Prop:nBlockBounds}
Suppose $(V,P)$ is mixing Markov chain. Then there exists $K >0$ such that for large $n$,
\begin{equation*}
 \frac{1}{3} \frac{1}{\Delta_n} \leq \mbar_n \leq m_n^* \leq K n \frac{1}{\Delta_n}.
\end{equation*}
\end{proposition}
\begin{proof}
 In the trivial case $|V|=1$, we have $\Delta_n = 1 = \mbar_n = m_n^*$, and the desired inequalities hold for all $n$. Now suppose that $|V|>1$.
 The estimate $\mbar_n \leq m_n^*$ is immediate from the definitions. Let us prove the other inequalities. Recall that $\mbar_n = \E_{\mu \times \mu}(M_n)$, where $M_n(x,y)$ is the first time $t$ such that $x_t^{t+n-1} = y_t^{t+n-1}$. Also,
\begin{align} \label{Eqn:Dodgers}
\begin{split} 
 \mu \times \mu ( M_n \leq k ) & = \mu \times \mu \Biggl( \bigcup_{t \leq k} \{ x_t^{t+n-1} = y_t^{t+n-1} \} \Biggr) \\
 & \leq \sum_{t=1}^k \mu \times \mu \biggl( x_t^{t+n-1} = y_t^{t+n-1} \biggr) \\
 & = k \mu \times \mu \biggl( x_1^n = y_1^n \biggr) \\
 & = k \Delta_n,
\end{split}
\end{align}
where we have used the translation invariance of $\mu \times \mu$. 
Let $I_n = \lfloor \frac{1}{\Delta_n} \rfloor$. Then by the tail-sum formula and (\ref{Eqn:Dodgers}), we have that
\begin{align} 
\begin{split} \label{Eqn:Dugout}
\mbar_n & =  \E_{\mu \times \mu}(M_n) \\
 & = \sum_{t >0} \mu \times \mu( M_n >t) \\
 & = \sum_{t > 0 } (1 - \mu \times \mu( M_n \leq t) ) \\
 & \geq \sum_{t =1}^{I_n} (1 - t \Delta_n) \\
 & \geq I_n - \Delta_n I_n (I_n+1)/2 \\
 & = I_n \Bigl( 1- (I_n+1)\Delta_n / 2\Bigr).
\end{split}
\end{align}
Then by the elementary inequalities 
\begin{equation*}
 \frac{1}{\Delta_n} - 1 \leq \biggl\lfloor \frac{1}{\Delta_n} \biggr\rfloor \leq \frac{1}{\Delta_n},
\end{equation*}
we see that $I_n \geq \frac{1}{\Delta_n} - 1 = \frac{1}{\Delta_n}(1-\Delta_n)$ and $(I_n+1)\Delta_n \leq 1+\Delta_n$. Using these inequalities in (\ref{Eqn:Dugout}), we obtain
\begin{equation} \label{Eqn:Bullpen}
 \mbar_n \geq I_n \Bigl( 1- (I_n + 1)\Delta_n / 2\Bigr) \geq  \frac{1}{\Delta_n}(1-\Delta_n)\Bigl( 1 - \frac{1}{2} - \Delta_n/2 \Bigr).
\end{equation}
Since $|V|>1$, we have that $\Delta_n$ tends to $0$ as $n$ tends to infinity (by Propositions \ref{Prop:DeltanLimit} and \ref{Prop:LnotZero}). Thus, for large $n$, (\ref{Eqn:Bullpen}) gives that 
\begin{equation*}
 \mbar_n \geq \frac{1}{3 \Delta_n}.
\end{equation*}

We now proceed to show that there exists $K >0$ such that for large $n$,
\begin{equation*}
  m_n^* \leq K n \frac{1}{\Delta_n}.
\end{equation*}
Fix $\epsilon \in (0,1)$. By Lemma \ref{Lemma:MeetingTimeBound}, there exists $T$ such that for any $u,v$ in $V_n$,
\begin{equation} \label{Eqn:Yankees}
 \P( m_n(u,v) \leq n+T) \geq (1-\epsilon) \Delta_n.
\end{equation}
Using (\ref{Eqn:Yankees}) and induction on $k$, one may check that
\begin{equation} \label{Eqn:Phillies}
 \P( m_n(u,v) > k(n+T) ) \leq (1-(1-\epsilon)\Delta_n)^k.
\end{equation}
Let $r_n = 1 - (1-\epsilon)\Delta_n$. Then tail-sum and (\ref{Eqn:Phillies}) give
\begin{align}
 \begin{split}
  \E( m_n(u,v) ) & = \sum_{t = 0}^{\infty} \P( m_n(u,v) > t) \\
    & \leq \sum_{t = 0}^{\infty} \P\biggl( m_n(u,v) > \biggl \lfloor \frac{t}{n+T} \biggr \rfloor (n+T) \biggr) \\
    & \leq \sum_{t = 0}^{\infty} r_n^{ \bigl\lfloor \frac{t}{n+T} \bigr\rfloor} \\
    & = \sum_{k = 0}^{\infty} (n+T) r_n^k  \\
    & =  \frac{n+T}{1-r_n} \\
    & = \frac{n+T}{(1-\epsilon) \Delta_n}.
 \end{split}
\end{align}
Thus with $K = \frac{1+T}{1-\epsilon}$, we have shown that
\begin{equation*}
 m_n^* \leq K n \frac{1}{\Delta_n},
\end{equation*}
as desired.
\end{proof}

With the above bounds in place, we are now in a position to characterize the exponential rate of growth of \textit{expected} meeting times and \textit{expected} coalescence times.

\begin{theorem} \label{Thm:Expectations}
Suppose $(V,P)$ is a mixing Markov chain. Then
\begin{align*}
 L(V,P) & = \lim_n - \frac{1}{n} \log \Delta_n \\
   & =  \lim_n \frac{1}{n} \log \E(C_n) \\
   & = \lim_n \frac{1}{n} \log \mbar_n \\
   & = \lim_n \frac{1}{n} \log m_n^*.
\end{align*}
In particular, Theorem \ref{Thm:ThirdMain} holds.
\end{theorem}
\begin{proof}
By Lemma \ref{Lemma:MeetingTimeLB}, Lemma \ref{Lemma:FirstMomentBound} and Proposition \ref{Prop:nBlockBounds}, there exist constants $K_1$ and $K_2$ such that
\begin{equation} \label{Eqn:AllBounds}
 \frac{1}{3} \frac{1}{\Delta_n} \leq \mbar_n \leq  m_n^* \leq \E(C_n) \leq K_1 n \, m_n^* \leq K_2 n^2 \frac{1}{\Delta_n}.
\end{equation}
Then by Proposition \ref{Prop:DeltanLimit} and the bounds in (\ref{Eqn:AllBounds}), we see that
\begin{align*}
 L(V,P) & = \lim_n - \frac{1}{n} \log \Delta_n \\
   & =  \lim_n \frac{1}{n} \log \E(C_n) \\
   & = \lim_n \frac{1}{n} \log \mbar_n \\
   & = \lim_n \frac{1}{n} \log m_n^*.
\end{align*}
\end{proof}

\subsection{Convergence in probability and almost surely}

We are now prepared to prove Theorems \ref{Thm:FirstMain} and \ref{Thm:FourthMain}, which characterize the exponential growth rate of coalescence and meeting times in probability and almost surely. The following proposition shows that it is highly unlikely for coalescence and meeting times to occur too soon. 

\begin{proposition} \label{Prop:TooEarly}
Suppose $(V,P)$ is a mixing Markov chain. For each $\epsilon >0$ there exists $\delta>0$ and $N_1$ such that if $n \geq N_1$, then
\begin{equation} \label{Eqn:RedSox1}
 \P\biggl( \frac{1}{n} \log C_n - L < - \epsilon \biggr) \leq \mu \times \mu \bigl( M_n  < e^{n(L-\epsilon)} \bigr) \leq e^{-\delta n}.
\end{equation}
\end{proposition}
\begin{proof}
Let $\epsilon >0$. Since $M_n \leq C_n$ (Lemma \ref{Lemma:MeetingTimeLB}), we have that
\begin{equation*}
 \biggl\{  \frac{1}{n} \log C_n - L < -\epsilon \biggr\} \subset \biggl\{ M_n < e^{n(L-\epsilon)} \biggr\}.
\end{equation*}
Furthermore,
\begin{align*}
 \mu \times \mu \bigl(  M_n < e^{n(L-\epsilon)} \bigr) & = \mu \times \mu \Biggl( \bigcup_{j=1}^{e^{n(L-\epsilon)}-1} \Bigl\{ x_j^{j+n-1} = y_j^{j+n-1} \Bigr\} \Biggr) \\
 & \leq \sum_{j=1}^{e^{n(L-\epsilon)}-1} \mu \times \mu \bigl( x_j^{j+n-1} = y_j^{j+n-1} \bigr) \\
 &  \leq \sum_{j=0}^{e^{n(L-\epsilon)}-1} \mu \times \mu \bigl( (\sigma \times \sigma)^{-j}\{x_1^{n} = y_1^{n}\} \bigr) \\
 & = e^{n(L-\epsilon)} \mu \times \mu \bigl( x_1^{n} = y_1^{n} \bigr) \\
 & = e^{n(L-\epsilon + \frac{1}{n} \log \Delta_n)},
\end{align*}
where we have used that $\mu$ is $\sigma$-invariant. Since $\lim_n \frac{1}{n} \log \Delta_n = -L$ (Proposition \ref{Prop:DeltanLimit}), we have shown that there exists $\delta >0$ and $N$ such that for $n \geq N$, it holds that
\begin{equation*} 
 \P\biggl( \frac{1}{n} \log C_n - L < - \epsilon \biggr) \leq \mu \times \mu \bigl( M_n  < e^{n(L-\epsilon)} \bigr) \leq e^{-\delta n}.
\end{equation*}
\end{proof}

By the above proposition, we have bounds on the probability of meeting or coalescing too soon. The following proposition gives similar bounds on the probability of meeting or coalescing too late. The basic idea is to use the moment bounds from Section \ref{Sect:Moments} in a second moment argument.

\begin{proposition} \label{Prop:TooLate}
Suppose $(V,P)$ is a mixing Markov chain. For each $\epsilon >0$, there exists $\delta >0$ and $N_2$ such that if $n\geq N_2$, then
\begin{equation} \label{Eqn:RedSox2}
  \mu \times \mu \biggl( \frac{1}{n} \log M_n > L+\epsilon \biggr) \leq \P \biggl( \frac{1}{n} \log C_n > L + \epsilon \biggr) \leq e^{-\delta n}.
\end{equation}
\end{proposition}
\begin{proof}
Let $\epsilon >0$. Since $M_n \leq C_n$ (Lemma \ref{Lemma:MeetingTimeLB}), we have that 
\begin{equation*}
 \mu \times \mu \biggl( \frac{1}{n} \log M_n > L+\epsilon \biggr) \leq \P \biggl( \frac{1}{n} \log C_n > L + \epsilon \biggr).
\end{equation*}
Choose $n$ large so that $L > \frac{1}{n} \log \E(C_n) - \epsilon/2$ (which is possible by Theorem \ref{Thm:Expectations}). Then
\begin{align*}
 \mu \times \mu \biggl( \frac{1}{n} \log M_n > L+\epsilon \biggr) & \leq \P \biggl( \frac{1}{n} \log C_n > L + \epsilon \biggr) \\
& \leq \P \biggl( \frac{1}{n} \log C_n > \frac{1}{n} \log \E(C_n) +  \epsilon/2 \biggr).
\end{align*}
Also,
\begin{align*}
 \P \biggl( \frac{1}{n} \log C_n > & \frac{1}{n} \log \E(C_n) +  \epsilon/2 \biggr) \\
 & = \P \biggl( C_n - \E(C_n) > (e^{n\epsilon/2} - 1)\E(C_n) \frac{\sqrt{\Var(C_n)}}{\sqrt{\Var(C_n)}} \biggr).
\end{align*}
By Chebyshev's Inequality, we have that
\begin{equation*}
 \P \biggl( C_n - \E(C_n) > (e^{n\epsilon/2} - 1)\E(C_n) \frac{\sqrt{\Var(C_n)}}{\sqrt{\Var(C_n)}} \biggr) \leq \frac{\Var(C_n)}{\E(C_n)^2} \frac{1}{(e^{n \epsilon/2}-1)^2}.
\end{equation*}
By Lemma \ref{Lemma:SecondMoment}, there exists a constant $K$ such that
\begin{equation*}
  \frac{\Var(C_n)}{\E(C_n)^2} \frac{1}{(e^{n \epsilon/2}-1)^2} \leq  \frac{Kn^2}{(e^{n \epsilon/2}-1)^2}.
\end{equation*}
Combining the above inequalities, we obtain that there exists $\delta >0$ such that for large $n$,
\begin{equation*} 
  \mu \times \mu \biggl( \frac{1}{n} \log M_n > L+\epsilon \biggr) \leq \P \biggl( \frac{1}{n} \log C_n > L + \epsilon \biggr) \leq e^{-\delta n}.
\end{equation*}
\end{proof}

Finally, we give proofs of Theorem \ref{Thm:FirstMain} and Theorem \ref{Thm:FourthMain}. With the exponential bounds from Propositions \ref{Prop:TooEarly} and \ref{Prop:TooLate} already in place, Theorem \ref{Thm:FirstMain} is immediate and the proof of Theorem \ref{Thm:FourthMain} amounts to an application of Borel-Cantelli.

\vspace{3mm}
\begin{pfofthmFirst}{}
By Equations (\ref{Eqn:RedSox1}) and (\ref{Eqn:RedSox2}), we have that $\frac{1}{n} \log C_n$ converges to $L(V,P)$ in probability with exponential rate in $n$. This completes the proof of Theorem \ref{Thm:FirstMain}.
\end{pfofthmFirst}

\vspace{3mm}

\begin{pfofthmFourth}{}
By Borel-Cantelli, we conclude from Equation (\ref{Eqn:RedSox1}) that 
\begin{equation*}
\liminf_n \frac{1}{n} \log M_n \geq L - \epsilon, \quad \mu \times \mu-a.s.
\end{equation*}
Since $\epsilon$ was arbitrary, we obtain that 
\begin{equation} \label{Eqn:Sabres}
\liminf_n \frac{1}{n} \log M_n \geq L , \quad \mu \times \mu-a.s.
\end{equation}
Similarly, Borel-Cantelli and Equation (\ref{Eqn:RedSox2}) yield that 
\begin{equation*}
 \limsup_n \frac{1}{n} \log M_n \leq L + \epsilon, \quad \mu \times \mu - a.s.
\end{equation*}
Since $\epsilon$ was arbitrary, we see that
\begin{equation*}
 \limsup_n \frac{1}{n} \log M_n \leq L, \quad \mu \times \mu - a.s.
\end{equation*}
Thus, we have shown that
\begin{equation*}
\lim_n \frac{1}{n} \log M_n = L, \quad \mu \times \mu - a.s.,
\end{equation*}
which completes the proof of Theorem \ref{Thm:FourthMain}.
\end{pfofthmFourth}

\subsection{Characterizing $L(V,P)$}

In this section we prove the three statements in Theorem \ref{Thm:SecondMain}, which characterizes the relationship between $L(V,P)$ and the entropy of $(V,P)$, which we denote by $h(V,P)$. This section relies heavily on the thermodynamic formalism for Markov chains in Section \ref{Sect:ThermFormalism}. First, we show that entropy gives an upper bound on $L(V,P)$.

\begin{proposition} \label{Prop:LeqEntropy}
Suppose $(V,P)$ is a mixing Markov chain. Then $0 \leq L(V,P) \leq h(V,P)$.
\end{proposition}
\begin{proof}
Recall that $Q(u,v) = P(u,v)^2$, $\lambda$ is the Perron eigenvalue of $Q$, and $L(V,P) = - \log \lambda$. Since $P$ is stochastic, we have that $Q$ is sub-stochastic. Therefore $\lambda$, defined as the Perron eigenvalue of $Q$, satisfies $0 \leq \lambda \leq 1$. Hence $L(V,P) = - \log \lambda \geq 0$. 

We now proceed to show that $L(V,P) \leq h(V,P)$.  Define the sequence of functions $g_n : V_n \to \R$ such that $g_n(u) = \pi_n(u) = \mu(u)$. Then
\begin{equation*}
\Delta_n = \sum_{u \in V_n} \mu(u)^2 = \E_{\pi_n}(g_n).
\end{equation*}
By Jensen's inequality, we have
\begin{equation*}
-\frac{1}{n} \log(\Delta_n)  = -\frac{1}{n} \log \E_{\pi_n}(g_n) \leq -\frac{1}{n}\E_{\pi_n} (\log g_n).
\end{equation*}
Then by Proposition \ref{Prop:DeltanLimit} and the Shannon-McMillan-Breiman Theorem, we obtain
\begin{equation*}
L(V,P) = \lim_n -\frac{1}{n} \log(\Delta_n) \leq \lim_n-\frac{1}{n}\E_{\pi_n} (\log g_n) = h(V,P).
\end{equation*}
\end{proof}

By the previous proposition, we know that $L(V,P) \geq 0$. The following proposition characterizes exactly when $L(V,P) = 0$.

\begin{proposition} \label{Prop:LnotZero}
 The chain $(V,P)$ is non-trivial (\textit{i.e.} $|V|>1$) if and only if $L(V,P) > 0$.
\end{proposition}
\begin{proof}
If $|V| = 1$, then $Q = P = (1)$, and therefore $L(V,P) = - \log 1 = 0$.

Now suppose that $(V,P)$ is non-trivial. It is well-known (see \cite{Abadi}, for example) that if $(V,P)$ is a non-trivial, mixing Markov chain, then there exists $\delta>0$ such that for large $n$, if $u \in V_n$, then $\mu(u) \leq e^{-\delta n}$. Thus, for large enough $n$, 
\begin{equation*}
\Delta_n = \sum_{u \in V_n} \mu(u)^2 \leq \max_{u \in V_n} \mu(u) \sum_{u \in V_n} \mu(u) = \max_{u \in V_n} \mu(u)  \leq e^{-\delta n},
\end{equation*}
and therefore
\begin{equation} \label{Eqn:Home}
 -\frac{1}{n} \log \Delta_n \geq \delta.
\end{equation}
By Proposition \ref{Prop:DeltanLimit} and (\ref{Eqn:Home}), we have 
\begin{equation*}
L(V,P) = \lim_n - \frac{1}{n} \log \Delta_n \geq \delta >0.
\end{equation*}
\end{proof}

By Proposition \ref{Prop:LeqEntropy}, we know that $L(V,P) \leq h(V,P)$. Now we characterize exactly when we have equality of these two quantities.

\begin{theorem} \label{Thm:EqEntropy}
 Suppose $(V,P)$ is a mixing Markov chain. Then $L(V,P) = h(V,P)$ if and only if $(V,P)$ is a measure of maximal entropy.
\end{theorem}
\begin{proof}
In this proof, we will have occasion to use the thermodynamic formalism for Markov chains presented in Section \ref{Sect:ThermFormalism}. Recall that the Markov measure $\mu$ corresponding to $(V,P)$ is the unique equilibrium state corresponding to the function $g : X \to \R$, where $g(x) = \log P(x_1,x_2)$, and we write $\mu = \mu_g$. Let $f : X \to \R$ be the function $f(x) = \log Q(x_1,x_2) = 2 \log P(x_1,x_2) = 2 g(x)$. Assume that $L(V,P) = h(V,P) = h(\mu)$. Then $\mathcal{P}(f) = \log \lambda = - L(V,P)$ (see Section \ref{Sect:ThermFormalism}), and using (\ref{Eqn:Entropy}), we have
\begin{align*}
\mathcal{P}(f) & = -L(V,P) \\
 & = -h(\mu) \\
 & = h(\mu) +2 \sum_{u,v} \pi(u) P(u,v) \log P(u,v) \\
 & = h(\mu) +\int f d\mu.
\end{align*}
By the uniqueness of the equilibrium state for $f$, we have that $\mu = \mu_f$. Hence we've shown that $\mu_g = \mu = \mu_f$. By a result of Parry and Tuncel (see Theorem \ref{Thm:ParryTuncel}), it follows that there exists a continuous function $k : X \to \R$ and a constant $c$ such that
\begin{equation*}
f = g +k - k\circ \sigma + c.
\end{equation*}
By rearranging this equation and using $g = f-g$, we see that
\begin{equation*}
g = f-g = k-k\circ\sigma+c.
\end{equation*}
Thus, by Theorem \ref{Thm:ParryTuncel}, we obtain that $\mu_g = \mu_0$. Recall that $\mu_0$ is defined as the measure of maximal entropy for $X$, and therefore  we have shown that $\mu$ is the measure of maximal entropy on $X$.

Let us now prove the reverse implication. Assume $\mu$ is the measure of maximal entropy on $X$. It is well-known (see \cite[Theorem 1.2]{Bowen}, for example) that since $\mu$ is the measure of maximal entropy, there exists a constant $K>0$ such that for each $u$ in $V_n$, 
\begin{equation*}
K^{-1} e^{-h(\mu) n} \leq \mu(u) \leq K e^{-h(\mu) n}.
\end{equation*}
Then
\begin{equation*}
K^{-1} e^{-h(\mu)n} \leq \min_{u \in V_n} \mu(u) \leq \Delta_n \leq \max_{u \in V_n} \mu(u) \leq K e^{-h(\mu)n}.
\end{equation*}
Therefore, using Proposition \ref{Prop:DeltanLimit}, we have that
\begin{equation*}
L(V,P) = \lim_n -\frac{1}{n}\log \Delta_n = h(\mu),
\end{equation*}
as desired.
\end{proof}

\section*{Acknowledgements}
KL gratefully acknowledges the support of the Duke PRUV program. KM was supported by National Science Foundation grant number 10-45153. The authors would like to thank Rick Durrett and Mike Boyle for productive conversations regarding this work, as well as the anonymous referee for helpful suggestions. 

\bibliographystyle{spmpsci}      
\bibliography{CoalescenceRefs}   

\end{document}